\documentclass{amsart}
\usepackage{amsmath,amssymb,latexsym}
\usepackage{amsthm}
\usepackage{amsfonts}
\usepackage{mathtools,caption}
\usepackage{autobreak}
\usepackage{bm}
\usepackage[OT2,OT1]{fontenc}

\usepackage{mleftright}
\usepackage{xparse}
\usepackage{mleftright}
\usepackage{bbm}
\usepackage{enumitem}

\usepackage{array}

\renewcommand{\arraystretch}{1.25}  

\usepackage{xcolor}
\usepackage[colorlinks=true, linkcolor=blue, citecolor=blue]{hyperref}

\usepackage{aliascnt}
\usepackage{cleveref}

\numberwithin{equation}{section}

\mathtoolsset{showonlyrefs=true}

\newcommand{\dsum}{\displaystyle\sum}

\newcommand{\set}[1]{\left\{ #1 \right\}}

\newcommand{\abs}[1]{\left \lvert #1 \right \rvert}
\newcommand{\leg}[2]{\left(\frac{#1}{#2}\right)}

\newcommand{\floor}[1]{\left\lfloor#1\right\rfloor} 
\newcommand{\inv}[1]{#1^{-1}}
\newcommand{\dprod}{\displaystyle\prod}

\makeatletter
\def\widebar{\accentset{{\cc@style\underline{\mskip10mu}}}}
\makeatother

\makeatletter
\def\wideubar{\underaccent{{\cc@style\underline{\mskip10mu}}}}
\makeatother

\def\x{^{\times}}
\def\inv{^{-1}}
\def\1{\mathbbm{1}}

\def\jac{\mathrm{Jac}}

\def\A{\mathbb{A}}

\def\C{\mathbb{C}}

\def\F{\mathbb{F}}

\def\Q{\mathbb{Q}}
\def\R{\mathbb{R}}

\def\Z{\mathbb{Z}}

\def\mf{\mathfrak}

\def\mc{\mathcal}

\DeclareMathOperator{\gal}{Gal}

\DeclareMathOperator{\ord}{ord}

\theoremstyle{definition}

\newtheorem{theoremeigo}{Theorem}[section]

\newaliascnt{definitioneigo}{theoremeigo}
\newtheorem{definitioneigo}[definitioneigo]{Definition}
\aliascntresetthe{definitioneigo}

\newaliascnt{lemmaeigo}{theoremeigo}
\newtheorem{lemmaeigo}[lemmaeigo]{Lemma}
\aliascntresetthe{lemmaeigo}

\newaliascnt{propositioneigo}{theoremeigo}
\newtheorem{propositioneigo}[propositioneigo]{Proposition}
\aliascntresetthe{propositioneigo}

\newaliascnt{corollaryeigo}{theoremeigo}
\newtheorem{corollaryeigo}[corollaryeigo]{Corollary}
\aliascntresetthe{corollaryeigo}

\newaliascnt{exampleeigo}{theoremeigo}
\newtheorem{exampleeigo}[exampleeigo]{Example}
\aliascntresetthe{exampleeigo}

\newaliascnt{remarkeigo}{theoremeigo}
\newtheorem{remarkeigo}[remarkeigo]{Remark}
\aliascntresetthe{remarkeigo}

\newaliascnt{noteeigo}{theoremeigo}
\newtheorem{noteeigo}[noteeigo]{Note}
\aliascntresetthe{noteeigo}

\newaliascnt{problemeigo}{theoremeigo}

\aliascntresetthe{problemeigo}

\newaliascnt{conjectureeigo}{theoremeigo}

\aliascntresetthe{conjectureeigo}

\newaliascnt{notationeigo}{theoremeigo}

\aliascntresetthe{notationeigo}

\crefname{theoremeigo}{theorem}{theorems}
\Crefname{theoremeigo}{Theorem}{Theorems}
\crefname{definitioneigo}{definition}{definitions}
\Crefname{definitioneigo}{Definition}{Definitions}
\crefname{lemmaeigo}{lemma}{lemmas}
\Crefname{lemmaeigo}{Lemma}{Lemmas}
\crefname{propositioneigo}{proposition}{propositions}
\Crefname{propositioneigo}{Proposition}{Propositions}
\crefname{corollaryeigo}{corollary}{corollaries}
\Crefname{corollaryeigo}{Corollary}{Corollaries}
\crefname{exampleeigo}{example}{examples}
\Crefname{exampleeigo}{Example}{Examples}
\crefname{remarkeigo}{remark}{remarks}
\Crefname{remarkeigo}{Remark}{Remarks}
\crefname{noteeigo}{note}{notes}
\Crefname{noteeigo}{Note}{Notes}
\crefname{problemeigo}{problem}{problems}
\Crefname{problemeigo}{Problem}{Problems}
\crefname{conjectureeigo}{conjecture}{conjectures}
\Crefname{conjectureeigo}{Conjecture}{Conjectures}
\crefname{notationeigo}{notation}{notations}
\Crefname{notationeigo}{Notation}{Notations}
\crefname{claimeigo}{claim}{claims}
\Crefname{claimeigo}{Claim}{Claims}

\theoremstyle{definition}

\makeatletter

\title{Root numbers for twisted Fermat quotient curves II}
\author{Ryosuke Yanagihara
}

\address{Ryosuke Yanagihara\\
Mathematical inst. Tohoku Univ.\\
6-3, Aoba, Aramaki, Aoba-ku, Sendai, 980-8578, JAPAN}
\email{yanagihara.ryosuke.t1@dc.tohoku.ac.jp}



\begin{document}

\begin{abstract}
    This is a sequel to the previous work of the author  \cite{yanagihara2025rootnumberstwistedfermat}. 
	Let $\ell$ be an odd prime, let $N \geq 1$ be an integer, and let $\delta \geq 1$ be an integer such that $\delta$ is $\ell^N$-th-power-free.
    Let $r,s,t>0$ be integers satisfying $r+s+t=\ell^N$.
    In \cite{yanagihara2025rootnumberstwistedfermat}, the author computed the root number of the twisted Fermat quotient curve $y^{\ell^N}=x^r(\delta-x)^s$ under the assumptions that $\ell\nmid rst$ and that $\ord_{\ell}(\delta)=0$ or $\ell\nmid \ord_{\ell}(\delta)$.
    In this paper, we study the case where the technical assumption
$\ell\nmid rst$ is dropped.
As one such case, we compute the root number when
$\ell^{N-1}\| r$ and $\ell\nmid st\delta$.
\end{abstract}

\maketitle
\tableofcontents

\section{Introduction}
The study of the root numbers of abelian varieties over number fields 
is a central topic in arithmetic geometry. 
Especially, CM abelian varieties are attractive and 
accessible examples to carry out the computation
(see \cite[Section 1]{shu2021root}). 

This paper is a sequel to our previous work \cite{yanagihara2025rootnumberstwistedfermat} and we continue to study the root numbers 
of the CM abelian varieties which are quotients of 
the Jacobian varieties of the twisted Fermat curves. 
To explain the main result, we prepare some notation. 

Let $\ell$ be an odd prime, let $N \geq 1$ be an integer, and let $\delta \geq 1$ be an integer such that $\delta$ is $\ell^N$-th-power-free (i.e., for every prime $p$, the integer $\delta$ is not divisible by $p^{\ell^N}$).

Let $r, s, t > 0$ be integers such that $r+s+t=\ell^N$, with $r=\ell^{N-1}r'$ $(\ell\nmid r')$ and $\ell\nmid st$, and let $\ell\nmid\delta$. Let $1 \leq i \leq N$ be an integer.
We define a family of curves $C_i^{(\delta; r_i, s_i, t_i)}$. In this introduction, let $\overline{a}$ denote the nonnegative remainder of an integer $a$ modulo $\ell^i$.

\begin{definitioneigo}
    First, we set the following integers $0\leq r_i, s_i,t_i<\ell^i$ as follows:
    $$(r_i,s_i,t_i):=
    \begin{dcases}
        (\overline{r}, \overline{s}, \overline{t})\ &(\overline{r}+\overline{s}+\overline{t}=\ell^i)\\
        (\ell^i-\overline{r}, \ell^i-\overline{s},\ell^i-\overline{t})\ &(\overline{r}+\overline{s}+\overline{t}=2\ell^i).
    \end{dcases}$$
Next, let $\delta_i:=\dprod_{p} p^{\overline{\ord_{p} (\delta)}}.$
Here, we note that $r_N=r,\ s_N=s,\ t_N=t,\ \delta_N=\delta$ and $\delta_i$ is $\ell^i$-th power free. 

Now, let  $F_{i}=F_{i,\delta_i}$ denote the nonsingular projective curve over $\Q$ defined by 
$$X^{\ell^i}+Y^{\ell^i} =\delta_i,$$ and we call it an \textit{$\ell^i$-th twisted Fermat curve}.

Let  $C_i = C_i^{(\delta; r_i, s_i, t_i)}$ denote the nonsingular projective curve over $\Q$ defined by 
$$y^{\ell^i} = x^{r_i}(\delta_i - x)^{s_i},$$ and we call it an \textit{$\ell^i$-th twisted Fermat quotient curve}.
\end{definitioneigo}

The $\ell^i$-th twisted Fermat quotient curve $C_i$ is not only a quotient of $F_{i}$ but also of $C_{i+1}$ over $\Q$ (see \cite[Introduction]{yanagihara2025rootnumberstwistedfermat} for details). Thus, there is a quotient map $\pi_i:C_{i+1}\longrightarrow
 C_i$ defined over $\Q$ for $i<N$ and we define the new part of the Jacobian variety $\jac(C_i)$ as
$$\jac(C_i)^{\rm{new}}:=\jac(C_i)/\pi^\ast_{i-1}(\jac(C_{i-1}))$$
 where $\jac(C_1)^{\rm{new}}:=\jac(C_1)$. We note that for \(i<N\), we have \(r_i=0\) and \(\ell\nmid s_i\). Hence \(C_i\), defined by \(y^{\ell^i}=(\delta_i-x)^{s_i}\), is rational. In particular, \(\operatorname{Jac}(C_i)=0\) for \(i<N\), and thus \(\operatorname{Jac}(C_N)^{\mathrm{new}}=\operatorname{Jac}(C_N)\).
 Now, we would like to recall how the previous paper \cite{yanagihara2025rootnumberstwistedfermat} proceeded.

    We first recall the Hecke character
$\phi_{\delta}^{(i)}$ associated with the new part $\jac(C_i)^{\rm{new}}$ of the Jacobian variety $\jac(C_i)$ for the Fermat quotient curve $C_{i}$ identified in \cite[Section~3]{yanagihara2025rootnumberstwistedfermat}:

\begin{definitioneigo}\label{hecke character}
    In the following, we only consider the case \(i=N\).
    We first consider a unitary Hecke character $\phi^{(i)}_{r_i,s_i,t_i}:\A^\times_{K_i}\to \C^\times$ of a cyclotomic field $K_i=\Q(\zeta_{\ell^i})$ as follows:
\begin{enumerate}
\item[(i)] for each finite prime $V$ (with the corresponding prime ideal $\frak p_V$) of ${K_i}$ not dividing $\ell {\delta_i}$ 
and each uniformizer $\pi_{i,V}$ of $K_{i,V}$, 
$$\phi^{(i)}_{r_i,s_i,t_i}(\pi_{i,V})=j^{(i)}_{r_i,s_i,t_i}(\mf{p}_V)\mathcal{N}(\frak p_V)^{-\frac{1}{2}}$$ 
where $$
j^{(i)}_{r_i,s_i,t_i}(\mf{p}_V) := -J\left(\leg{\cdot}{\mf{p}_V}_{\ell^i}^{r_i}, \leg{\cdot}{\mf{p}_V}_{\ell^i}^{s_i}\right);
$$
\item[(ii)] the infinite component $\phi^{(i)}_{r_i,s_i,t_i,\infty}:({K_i}\otimes_\Q\R)^\times\longrightarrow 
\C^\times$ of $\phi^{(i)}_{r_i,s_i,t_i}$ is defined by  
$$\phi^{(i)}_{r_i,s_i,t_i,\infty}(x\otimes 1)=\prod_{\sigma\in \Phi^{-1}_i}\left(\frac{\sigma(x)}{|\sigma(x)|}\right)\inv$$ 
where the set $\Phi^{-1}_i$ is the CM type of ${K_i}$ defined by 
$$\Phi_i\inv := \left\{ h\inv \in (\Z/\ell^i\Z)\x \Bigg| \left\{\frac{r_i h}{\ell^i}\right\} + \left\{\frac{s_i h}{\ell^i}\right\} + \left\{\frac{t_i h}{\ell^i}\right\} = 1 \right\}
$$
under the canonical identification ${\rm Gal}({K_i}/\Q)\simeq (\Z/\ell^i\Z)^\times$. 
\end{enumerate}

Let $\chi^{(i)}_{\delta_i}:\A^\times_{K_i}\longrightarrow \C^\times,\ 
x\mapsto \dprod_{V\nmid \infty} (x_V,\delta_i)_{\ell^i, K_{i,V}}$ so that $\chi^{(i)}_{\delta_i,V}(\iota_V(\pi_{i,V}))=\left(\frac{\delta_i}{\frak p_V}\right)_{\ell^i}$ for any finite place $V$ of 
${K_i}$ not dividing $\ell {\delta_i}$ where $\iota_V :K_{i,V}\to \A^\times_{K_i}$ is the canonical embedding (cf. \cite[Chapter V, Proposition 3.1]{neukirch2013algebraic}).  We twist $\phi^{(i)}_{r_i,s_i,t_i}$ by ${\chi^{(i)}_{\delta_i}}^{r_i+s_i}$:
$$\phi^{(i)}_{\delta_i,r_i,s_i,t_i}:=
{\chi^{(i)}_{\delta_i}}^{r_i+s_i} \phi^{(i)}_{r_i,s_i,t_i}.
$$
\end{definitioneigo}
    
    In \cite{yanagihara2025rootnumberstwistedfermat}, the author computed the root number of the curve $C_N=C_N^{(\delta,r,s,t)}$ under the assumptions that $\ell\nmid rst$ and that either $\ord_{\ell}(\delta)=0$ or $\ell\nmid \ord_{\ell}(\delta)$.
    In this paper, we investigate what happens when we weaken the technical assumption $\ell\nmid rst$.
    
    This is not merely a matter of weakening assumptions; it can also lead to
arithmetically and geometrically interesting situations.
For example, when $N\geq 2$ and
$(r,s,t)=(\ell^{N-1}(\ell-1),1,\ell^{N-1}-1)$, the new part
$\jac(C_{N})^{\mathrm{new}}$
of the Jacobian variety $\jac(C_{N})$ coincides with
$\jac(C_{N})$ itself, which is an abelian variety with CM by
$\Q(\zeta_{\ell^N})$.
Furthermore, when $N=2$, one can show that
$\jac(C_{2})^{\mathrm{new}}=\jac(C_{2})$ is isomorphic over $\Q$
to the Jacobian of the non-singular projective plane 
curve $$D_{\delta}:\delta^{s+2t}U W^\ell+U^{\ell+1}=\delta V^{\ell}W.$$ 
Thus, the arithmetic study of $\jac(C_{2})^{\mathrm{new}}$
reduces to that of $D_\delta$. In particular, 
the case where $\ell=3$ has been handled in 
\cite{bouw2020conductor} for the arithmetic study of the Picard curves. 
We explain these results in Appendix~\ref{Appendix1}.

    As one such case, we compute the root number
$W(\phi_{\delta}^{(N)})$ of the Hecke character
$\phi_{\delta}^{(N)}$ in the case where $\ell^{N-1}\| r$ and $\ell\nmid st\delta$.

As the main theorem in this paper, we prove the following theorem:
 
\begin{theoremeigo}\label{main}
Write
$$r^r s^s (\ell^N-t)^t \delta^{r+s}e = \epsilon \ell^{b}(1 + c)\ (\epsilon \in \mu_{\ell-1}(\Q_{\ell}),\ b \in \Z,\ c \in \ell \Z_{\ell})$$
where $e=\begin{dcases}
    1+\frac{3^N}{r'} &(\ell=3,\ N\geq 2)\\
    1 &(\text{otherwise})
\end{dcases}$. Let $w=\min \{\ord_{\ell}(b),\ord_{\ell}(c)\}$.

Then, the root number of $\phi_{\delta}^{(N)}$ is
$$W(\phi_{\delta}^{(N)})=\dprod_{\text{$p\le \infty$}}W_p(\phi_{\delta}^{(N)})$$
where, for $p\neq \ell$,
$$W_p(\phi_{\delta}^{(N)})
=\begin{dcases}
	i^{-\frac{\ell^{N-1}(\ell-1)}{2}}\ &(p=\infty)\\
	\leg{p}{\ell}\ &(p\mid \delta)\\
	1\ &(p\nmid \delta),
\end{dcases} $$
and
\begin{align}
        W_{\ell}(\phi_{\delta}^{(N)})
        =\begin{dcases}
            -\leg{r'st\left\{(N-1)r'+\dfrac{c}{\ell^w}\right\}}{\ell}i^{\frac{\ell^{N-1}(\ell-1)}{2}} & (\ord_\ell(b + c) = w=N-1) \\
            -\leg{r'st\dfrac{c}{\ell^w}}{\ell}i^{\frac{\ell^{N-1}(\ell-1)}{2}} & \text{($\ord_\ell(b + c) = w<N-1$ or $w = N = \ord_\ell(c)$)} \\
            -\leg{-r'st\dfrac{c}{\ell^w}}{\ell}i^{\frac{\ell^{N-1}(\ell-1)}{2}} & \text{($1 \leq w < N$ and\ $\ord_\ell(b + c) >w$)} \\
            \leg{2}{\ell}i^{\frac{\ell^{N-1}(\ell-1)}{2}} &(\text{otherwise}).
        \end{dcases}
    \end{align}
\end{theoremeigo}

The calculation strategy here is almost the same as the previous work of the author \cite{yanagihara2025rootnumberstwistedfermat}; that is, first, we calculate local root numbers in the ramified primes by the Rohrlich formula \cite[Proposition 2]{rohrlich1992root}, and second, we compute two kinds of special values of the $\ell^N$ -th Hilbert symbol: a power of $(1+\pi^{f-1},\ell)_{\ell}$ and $(1+\pi^{f-1},1+c)_{\ell^N}$ for a prime element $\pi$, $c\in\ell\Z_{\ell}$, and a conductor $f$ of a ramified Kummer character.

What differs in our case is the calculation of the Hilbert symbol. 
Fortunately, a power of $(1+\pi^{f-1},\ell)_{\ell}$ follows from the previous work of the author, but the term $(1+\pi^{f-1},1+c)_{\ell^N}$ requires a new argument involving Fleck numbers. We found a somewhat general supercongruence for a sum of Fleck numbers $J(N,f)$ when $f$ is a conductor of a ramified Kummer character (\Cref{true conjecture}), and this allows us to compute it.

We organize this paper as follows. In Section 2, we review the computations in the author's previous work and compute a wildly ramified Hilbert symbol. In Section 3, we combine these calculations to compute the root number. In Section 4, we present some concrete examples and explain a structure of the new parts of the Jacobian varieties of the twisted Fermat quotients under our setting.\\

\textbf{Acknowledgments}

The author expresses deep gratitude to his supervisor, Takuya Yamauchi, for his careful and patient guidance. 
The author is supported by JST SPRING, Grant Number JPMJSP2114.\\

\textbf{Statement on the Use of AI}

The author acknowledges extensive interactions with
GPT-6.1 Sol during the exploratory and proof-development stages of this
work. All AI-generated suggestions were 
independently verified by the author, who assumes full responsibility for the
mathematical content. \\

\textbf{Notation.}
\begin{itemize}
	\item Let $\ell$ be an odd prime. Let $N \geq 1$ be an integer, and $\delta \geq 1$ be an $\ell^N$-th power free integer (i.e. for any prime $p$, $\delta$ is not divisible by $p^{\ell^N}$), satisfying $\ord_{\ell}(\delta) = 0$.
        \item For a prime $p$ and integers $k,\ n$, $p^k\| n$ means $p^k\mid n$ and $p^{k+1}\nmid n$.
	\item Let $\zeta_m = e^{2\pi i / m}$, and in particular, $\zeta := \zeta_{\ell^N}$ for simplicity, and let $K = \Q(\zeta)$ and $F = \Q(\zeta + \zeta^{-1})$.
        \item Let $\mc{O}_K$ (resp. $\mc{O}_F$) be the ring of integers of $K$ (resp. $F$). For an ideal $\mf{a}$ of the ring of integers $\mc{O}_L$ of some number field $L$, $\mc{N}(\mf{a})$ denotes its absolute norm (i.e. $\mc{N}(\mf{a})=\# (\mc{O}_L/\mf{a})$). 
        \item For a character $\chi$ of the multiplicative group of a finite field $\F_q$, we define its extension to $\F_q$ as usual:
        $$\chi(0):=
        \begin{dcases}
            1 &(\chi=1)\\
            0 &(\chi\neq 1).
        \end{dcases}$$
        \item For characters $\chi$ and $\psi$ of the multiplicative group of a finite field $\F_q$, we write their Jacobi sum as
$$
J(\chi,\psi) = \sum_{a \in \F_q} \chi(a) \psi(1-a).
$$
	\item Let $\pi = \pi_K = \zeta - \zeta^{-1}$ and $\pi' = 1 - \zeta$. We note that $\pi$ is a prime element of $K$ and satisfies $\pi_F := \pi^2 \in \mathcal{O}_F$.
	\item For a local field $M/\Q_p$ containing $n$-th primitive roots of unity, let $(a, b)_n=(a, b)_{n,M}$ be the $n$-th Hilbert symbol of $a,b\in M\x$ defined as
    $$(a, b)_n:=\frac{(a,M(\sqrt[n]{b})/M)(\sqrt[n]{b})}{\sqrt[n]{b}}$$
    where $(\cdot ,M(\sqrt[n]{b})/M):M\x\to\gal(M(\sqrt[n]{b})/M)$
    is the local reciprocity map.
    We write $[a, b]_n \in \Z / n\Z$ for the residue class such that $(a, b)_n = \zeta_n^{[a, b]_n}$.
	\item For $n,r\in\Z_{\geq 0}$ with $\ n\geq r$, let $\displaystyle\binom{n}{r}$ be the binomial coefficient and it is extended to all integers $r$ by $\displaystyle\binom{n}{r} := 0$  if $n < r$ or $r < 0$.
    \item For abelian varieties $A,B$ defined over $\Q$, we write $A\stackrel{\Q}{\sim}B$ if $A$ is isogenous to $B$ over $\Q$.
    \item For a real number $a$, $\lfloor a\rfloor$, $\lceil a\rceil$, and
$\{a\}:=a-\lfloor a\rfloor$ denote the integer part, the ceiling,
and the fractional part of $a$, respectively.
    \item For a proposition $P(n)$ of an integer variable $n$, $\1_{P(n)}$ is the indicator function. That is, it is $1$ if $P(n)$ is true and it is $0$ if $P(n)$ is false. 
    \item For a power series $f(X)\in\Z_{\ell}[[X]]$, we write $[X^n]f(X)$ for the coefficient of $X^n$.
\end{itemize}

\section{Calculation of the Hilbert symbols}

We first review the computations in the preceding work \cite{yanagihara2025rootnumberstwistedfermat}, and clarify how the present setting differs from that of it.

\subsection{\ Outline of calculations}
Let $\ell$ be an odd prime, and let $r,s,t,N$ be positive integers such that $r+s+t=\ell^N$, with $r=\ell^{N-1}r'$ $(\ell\nmid r')$ and $\ell\nmid st$. Let $\delta \geq 1$ be an integer such that $\delta$ is $\ell^N$-th-power-free and $\ell\nmid\delta$.

Let $\zeta=\zeta_{\ell^N},\ K=\Q(\zeta)$, and $\pi=\pi_K=\zeta-\zeta^{-1}$. $\pi$ is a prime of $K$ above $\ell$.
We consider the Hilbert symbol $(\cdot,\cdot)_{\ell^N}$ over $K_{\pi}=\Q_{\ell}(\zeta)$.

Let $\phi_{\delta}$ be the Hecke character defined in the introduction (\Cref{hecke character}).
Its root number is
$$
W(\phi_{\delta}) = \prod_{p\leq \infty} W_p(\phi_{\delta})
$$
where for each prime $p$, $W_p(\phi_{\delta}) := \prod_{V \mid p} W(\phi_{\delta, V})$ is the product of the local root numbers.
In \cite{yanagihara2025rootnumberstwistedfermat}, the following result was obtained without using assumptions on $r$ and $\delta$:

\begin{propositioneigo}(\cite[Proposition 4.6]{yanagihara2025rootnumberstwistedfermat})
For a prime $p\neq \ell$, we have
    \[
    W_p(\phi_{\delta}) =
    \begin{dcases}
        i^{-\frac{\ell^{N-1}(\ell-1)}{2}} & (p = \infty) \\
        1 & (p \nmid \delta) \\
        \leg{p}{\ell} & (p \mid \delta).
    \end{dcases}
    \]
\end{propositioneigo}

\begin{propositioneigo}(\cite[Section 4.3.1]{yanagihara2025rootnumberstwistedfermat})
Let $F$ be the maximal totally real subfield of $K$, and let $\eta$ be the quadratic Hecke character corresponding to the quadratic extension $K/F$ via class field theory.
For the prime element $\pi_F=\pi^2$, we have
    \[
W(\eta_{\pi_F}) = i^{\frac{\ell^{N-1}(\ell-1)}{2}}.
\]
\end{propositioneigo}
Therefore, we use Rohrlich's formula \cite[Proposition 2]{rohrlich1992root} to compute the relative root number $W(\phi_{\delta,\pi}, \eta_{\pi_F})$.
For this purpose, we need to compute the following Hilbert symbols (cf. \cite[Section 4.3.3]{yanagihara2025rootnumberstwistedfermat}).

We recall Sharifi's formula.
\begin{theoremeigo}[Sharifi's formula for $\Q_{\ell}\x$ {\normalfont\cite[Theorem 8]{sharifi2001norm}}]\label{sharifi}
    Let $a\in\Q_{\ell}\x$ and write $a = \epsilon \ell^b(1 + c)$ (where $\epsilon \in \mu_{\ell-1}(\Q_\ell)$, $b \in \Z$, and $c \in \ell \Z_\ell$).
    Let $w=w(a):=\min\{\ord_\ell(b),\ord_\ell(c)\}$. Then the conductor exponent of $(\cdot,a)_{\ell^N}$ is 
    \[
    f(a)=
    \begin{dcases}
        \ell^{N-1}(\ell+1) & (w = 0) \\
        2\ell^{N-w} & (1 \leq w < N,\ \ord_\ell(b + c) = w) \\
        \ell^{N-w-1}(\ell+1) & (1 \leq w < N,\ \ord_\ell(b + c) > w) \\
        2 & (w = N = \ord_\ell(c)) \\
        0 & (\text{otherwise}).
    \end{dcases}
    \]
\end{theoremeigo}

Let $A=r^rs^s(\ell^N-t)^t\delta^{r+s}e$ where $e=\begin{dcases}
    1+\frac{3^N}{r'} &(\ell=3,\ N\geq 2)\\
    1 &(\text{otherwise})
\end{dcases}$. Then we have $b:=\ord_{\ell}(A)=\ell^{N-1}(N-1)r'$.
Let $f=f(A)$ and we want to compute $(1+\pi^{f-1},A)_{\ell^N}$ when $f>0$.

    Moreover,
\begin{align}
    (1+\pi^{f-1},A)_{\ell^N}
&=(1+\pi^{f-1},\epsilon)_{\ell^N}(1+\pi^{f-1},\ell)_{\ell^N}^b (1+\pi^{f-1},1+c)_{\ell^N}\\
&=(1+\pi^{f-1},\ell)_{\ell}^{(N-1)r'} (1+\pi^{f-1},1+c)_{\ell^N}.
\end{align}
We compute these two factors $(1+\pi^{f-1},\ell)_{\ell}^{(N-1)r'}$ and $(1+\pi^{f-1},1+c)_{\ell^N}$ below.

\subsection{\ Calculation of a term involving $(1+\pi^{f-1},\ell)_{\ell}$}
We compute the first factor $(1+\pi^{f-1},\ell)_{\ell}^{(N-1)r'}$ in this section.
First, notice that we need not consider the case $N=1$, and we assume $N\geq 2$.
The conductor exponent of $(\cdot,\ell)_{\ell,\Q_{\ell}(\zeta_{\ell^N})}=(\cdot,\ell^{\ell^{N-1}})_{\ell^N,\Q_{\ell}(\zeta_{\ell^N})}$ is $f(\ell^{\ell^{N-1}})=2\ell$ by the Sharifi formula (Be careful not to confuse this with $(\cdot,\ell)_{\ell,\Q_{\ell}(\zeta_{\ell})}$ having a different base field). Therefore, the case where $f-1\geq f(\ell^{\ell^{N-1}})$ does not hold is only in the following cases:
\begin{itemize}
    \item (1) $w = N = \ord_\ell(c)$
    \item (2) $w=N-1$
\end{itemize}

We note that since the assumption $w=N$ in case (1) forces $\ell\mid N-1$, we only need to consider case (2), and 
\[
    f=
    \begin{dcases}
        2\ell & (\ord_\ell(b + c) = w) \\
        \ell+1 & (\ord_\ell(b + c) > w)
    \end{dcases}
    \]
then. In the following, we calculate the term for each case.



\begin{propositioneigo}
    $$(1+\pi^{\ell},\ell)_{\ell}=1.$$
\end{propositioneigo}
\begin{proof}
    Since $(1+\pi)^{\ell}-(1+\pi^{\ell})\in \pi^{\varphi(\ell^N)+1}\mc{O}_{K_{\pi}}$
    and $\varphi(\ell^N)+1\geq 2\ell=f(\ell^{\ell^{N-1}})$, we have
    $$(1+\pi^{\ell},\ell)_{\ell}
    =(1+\pi^{\ell},\ell^{\ell^{N-1}})_{\ell^N}
    =((1+\pi)^{\ell},\ell^{\ell^{N-1}})_{\ell^N}
    =1.$$
\end{proof}

\begin{propositioneigo}
    $$(1+\pi^{2\ell-1},\ell)_{\ell}=\zeta_{\ell}^{-2}.$$
\end{propositioneigo}
\begin{proof}
    Let $L:=\ell^{\ell^{N-1}}(1-\ell^{N-1})$. Since $2\ell-1\geq f(L)=\ell+1$, $(1+\pi^{2\ell-1},L)_{\ell^N}=1$.

    By the way, since $f(1-\ell^{N-1})=2\ell$, we know that $(1+\pi^{2\ell-1},1-\ell^{N-1})_{\ell^N}=\zeta_{\ell}^{2}$ by \cite[Proposition 4.16]{yanagihara2025rootnumberstwistedfermat}.

    Thus, we have
    $$(1+\pi^{2\ell-1},\ell)_{\ell}
    =(1+\pi^{2\ell-1},\ell^{\ell^{N-1}})_{\ell^N}
    =(1+\pi^{2\ell-1},L)_{\ell^N}\cdot (1+\pi^{2\ell-1},1-\ell^{N-1})_{\ell^N}\inv=\zeta_{\ell}^{-2}.$$
\end{proof}

Organizing above results, we have
$$(1+\pi^{f-1},\ell)_{\ell}^{(N-1)r'}
=\begin{dcases}
    \zeta_{\ell}^{-2(N-1)r'} & (\ord_\ell(b + c) = w=N-1) \\
    1 & (\text{otherwise}).
\end{dcases}$$

\subsection{Preparation for some congruences for Fleck numbers}

Let $p$ be an odd prime number through this section.

We recall that alternating binomial sums of the following form
$$C_n \binom{M}{r}=\sum_{\substack{k \equiv r\pmod{n}}}(-1)^k\binom{M}{k} $$
where $M\geq 0,n>0$ are integers and $r$ is an integer (Notice that $k$ runs in fact among $0\leq k\leq M$ by the notation of binomial coefficients and this is a finite sum) are called \textit{Fleck numbers} (\cite[Introduction (1.3)]{lettington2011fleck}), and they are of interest to have some nontrivial congruences (See \cite[Introduction]{Sun2006}). We note that we can put a rational number into $r$ when it is well-defined in mod $n$.

We need some congruences for Fleck numbers to compute the second factor $(1+\pi^{f-1},1+c)_{\ell^N}$.

In \cite[Proposition 4.10]{yanagihara2025rootnumberstwistedfermat}, we introduced a $p$-adic valued function $J(n,f)$, which is congruent to a linear combination of Fleck numbers (of course, it is a finite sum) mod $p^N$:
\begin{definitioneigo}\label{J}
    For a positive integer $n\geq 1$ and an integer $f>1$, we define the function $J(n,f)$ as follows:

    For integers $u\in\set{0,1}$ and $i\geq 0$, let $M_{i}=(f-1)(i+1)-1$ and $$U^{(u)}_i:=C_{p^{n-u}}\binom{M_i}{\frac{M_i+1}{2}}=\dsum_{k\equiv \frac{M_i+1}{2} \pmod{p^{n-u}}} (-1)^k \binom{M_i}{k}.$$
    Then, we define
	$$J(n,f):=\sum_{i\geq 0\atop\text{$i$:even}}
\left(p U^{(0)}_i-U^{(1)}_i\right)\in\Z_{p}.$$
\end{definitioneigo}

We especially consider a certain kind of Fleck numbers $\mc{F}_s(M)$;
\begin{definitioneigo}
    We write 
$$\mc{F}_s(M):=C_{p^s}\binom{M-1}{\frac{M}{2}}=\sum_{k\equiv \frac{M}{2}\pmod{p^s}} (-1)^k \binom{M-1}{k}$$
for integers $M>0,s\geq 0$. 

\end{definitioneigo}

\begin{noteeigo}
    When $M$ is odd, notice the following:
    
    We note that the index condition $k\equiv \dfrac{M}{2} \pmod{p^s}$ is equivalent to $k\equiv r_0 \pmod{p^s}$ where we write $r_0:=\dfrac{M-p^s}{2}$.
    
    We notice that $\mc{F}_s(M)=0$ when $M< p^s$.
    Furthermore, for $p^s\leq M<3p^s$, we have $$\mc{F}_s(M)=\sum_{k\equiv r_0\pmod{p^s}} (-1)^k \binom{M-1}{k}=(-1)^{r_0}\left\{\binom{M-1}{r_0}-\binom{M-1}{r_0-1}\right\}=(-1)^{r_0} \frac{p^s}{M} \binom{M}{r_0}.\cdots (\ast)$$
\end{noteeigo}

We evaluate the $p$-adic order of this.
\begin{lemmaeigo}\label{order of central Fleck}
    

    $$\ord_{p} (\mc{F}_s(M))\geq s-\ord_p(M)-1+\sum_{i=s}^{\infty} \floor{\frac{M}{p^i}}.$$
\end{lemmaeigo}
\begin{proof}
    Since the $s=0$ case is trivial, we may assume $s>0$.

    \begin{align}
        \mc{F}_s(M)
        &=\sum_{k\equiv r_0\pmod{p^s}} (-1)^k \binom{M-1}{k}\\
        &=\frac{1}{M}\sum_{k\equiv r_0\pmod{p^s}} (-1)^k \{(M-r_0)+(r_0-k)\}\binom{M}{k}\\
        &=-\frac{p^s}{M}\sum_{k\equiv r_0\pmod{p^s}} (-1)^k \frac{k-r_0}{p^s}\binom{M}{k},
    \end{align}
    noticing that $\dsum_{k\equiv r_0\pmod{p^s}} (-1)^k \binom{M}{k}=0$ since $M$ is odd.

    Now, \cite[Theorem 1.1]{sun--davis2007combinatorial} proves the claim.
\end{proof}

We prove a somewhat general formula for $\mc{F}_m(M)$;

\begin{propositioneigo}\label{decent}
    Let $m\geq 0$ be an integer, and let $D\geq p^m$ and $d > 0$ be odd integers with $p\nmid D$.
    Then, 
    $$\mc{F}_{m+t}(Dd)\equiv
    \begin{dcases}
        \mc{F}_{m}(D) \pmod{p^{m+1}} &(d=p^t)\\
        0 \pmod{p^{m+1}} &(d\neq p^t)
    \end{dcases}$$
    holds for any integer $t\geq 0$. 
    Moreover, when $D\geq 3p^m$ and $p^t\mid d$, $$\mc{F}_{m+t}(Dd)\equiv 0 \pmod{p^{m+2}}.$$
\end{propositioneigo}
\begin{proof}
    When $p^t\nmid d$, we may assume that $Dd\geq p^{m+t}$, and $\ord_p(\mc{F}_{m+t}(Dd))\geq m+1$ by \Cref{order of central Fleck}. 
    
    We may assume $p^t\mid d$ and write $d=p^{t}h$.

    By \Cref{order of central Fleck}, $$\ord_p(\mc{F}_{m+t}(Dd))\geq m-\ord_p(h)+\floor{\frac{Dh}{p^m}}-1.\cdots (\ast\ast)$$ 

    First, we have $\ord_p(\mc{F}_{m+t}(Dd))\geq m+2$ from $(\ast\ast)$ when $D\geq 3p^m$. 

\textbf{Case $h>1$.}
The right hand term of $(\ast\ast)$ is $\geq m+1$ since $h$ is odd and is $\geq \ord_p(h)+2$.

\textbf{Case $h=1$.}
When $p^m\leq D<3p^m$, ($\ast$) and Lucas' theorem prove the assertion.
\end{proof}

\begin{propositioneigo}\label{terminal}
    $$\mc{F}_{m+1}(p^{m}(p+1)-1)\equiv -p^{m+1} \pmod{p^{m+2}}$$
    for integers $m\geq 1$ and $\mc{F}_{1}(p)=1$ (this is the $m=0$ case).
\end{propositioneigo}
\begin{proof}
    This is an immediate consequence of ($\ast$) and Lucas' theorem.
\end{proof}

\begin{corollaryeigo}\label{conjecture}
Let $0\leq m\leq N-1$ be an integer, and $i\geq 0$ be an even integer, $u\in\{0,1\}$, and set $n=N,\ f=p^{m}(p+1)$.

\textbf{(a) The case $m\leq N-2$.} If $m\geq 1$, then,
\[
U_i^{(u)}
\equiv
\begin{dcases}
-p^{m+1}\pmod{p^{m+2}} & (i=p^{N-u-m-1}-1)\\
0\pmod{p^{m+2}} & (i\neq p^{N-u-m-1}-1).
\end{dcases}
\]
If $m=0$, then
\[
U_i^{(u)}
\equiv
\begin{dcases}
1\pmod{p} & (i=p^{N-u-1}-1)\\
0\pmod{p} & (i\neq p^{N-u-1}-1).
\end{dcases}
\]
Consequently,
    
    \begin{align}
  J(N,p^{m} (p+1))
  &\equiv
  \left\{
  \begin{array}{lll}
    p^{m+1} &\pmod{p^{m+2}} &(m\geq 1)\\
        -1 &\pmod{p} &(m=0).
  \end{array}
  \right.
\end{align}

\textbf{(b) The case $m= N-1$.} If $N\geq 2$, then, 
\[
U_i^{(u)}
\equiv 0 \pmod{p^{N+u}}.
\]
If $N=1$, then
\[
U_i^{(u)}
\equiv
\begin{dcases}
1\pmod{p^{1+u}} & (i=u=0)\\
0\pmod{p^{1+u}} & (\text{otherwise}).
\end{dcases}
\]
Consequently,
    
    \begin{align}
  J(N,p^{N-1} (p+1))
  &\equiv
  \left\{
  \begin{array}{lll}
    0 &\pmod{p^{N+1}} &(N\geq 2)\\
        p &\pmod{p^2} &(N=1).
  \end{array}
  \right.
\end{align}

\end{corollaryeigo}
\begin{proof}
(a) Suppose first that $m\geq 1$. Since $M_i=(p^{m}(p+1)-1)(i+1)-1$,
$U_i^{(u)}$ is the sum $\mc{F}_{N-u}((f-1)(i+1))$. Hence, by \Cref{decent,terminal},
\[
U_i^{(u)}
\equiv
\begin{dcases}
-p^{m+1}\pmod{p^{m+2}} & (i+1=p^{N-u-m-1})\\
0\pmod{p^{m+2}} & (i+1\neq p^{N-u-m-1}).
\end{dcases}
\]
Therefore,
\[
J(N,p^{m}(p+1))
=\sum_{\substack{i\geq0\\i:\mathrm{even}}}
(p U_i^{(0)}-U_i^{(1)})
\equiv p^{m+1}\pmod{p^{m+2}},
\]
and the first case follows.

Suppose next that $m=0$. Then $f=p+1$ and $M_i=p(i+1)-1$. By \Cref{decent,terminal},
\[
U_i^{(u)}
\equiv
\begin{dcases}
1\pmod{p} & (i=p^{N-u-1}-1)\\
0\pmod{p} & (i\neq p^{N-u-1}-1),
\end{dcases}
\]
noticing that the sum is empty when $i+1<p^{N-u-1}$.
Hence,
\[
J(N,p+1)
=\sum_{\substack{i\geq0\\i:\mathrm{even}}}
(p U_i^{(0)}-U_i^{(1)})
\equiv-1\pmod{p}.
\]
Thus, the result follows.

    (b) This is a straightforward consequence of \Cref{decent}.
\end{proof}
\begin{remarkeigo}
    We notice that \cite[Theorem 4.8]{yanagihara2025rootnumberstwistedfermat} also can be obtained by \Cref{decent} in the same way.
\end{remarkeigo}

Thus, we have the general formula for supercongruence for $J(N,f)$ when $f$ is a conductor exponent of a ramified Kummer character $(\cdot,a)_{\ell^N}$ ($a\in\Q_{\ell}\x$).

\begin{theoremeigo}\label{true conjecture}
Let
\[
a=\epsilon p^b(1+c)\in\Q_p^\times
\qquad
\left(
\epsilon\in\mu_{p-1}(\Q_p),\
b\in\Z,\
c\in p\Z_p
\right),
\]
and $w(a):=\min\{\ord_p(b),\ord_p(c)\}$, and $f(a)$ be the one given in \Cref{sharifi}. Then, we have
\[
\frac{(f(a)-1)J(N,f(a))}{p^{N-w(a)}}
\equiv \begin{dcases}
    -1 \pmod{p}\ &(0< w(a)\leq N)\\
    0 \pmod{p}\ &(w(a)=0)
\end{dcases}
\]
when $f(a)>0$.
\end{theoremeigo}
\begin{proof}
    Combine \Cref{conjecture} with \cite[Theorem 4.12]{yanagihara2025rootnumberstwistedfermat}.
\end{proof}

\subsection{\ Calculation of $(1+\pi^{f-1},1+c)_{\ell^N}$}\label{c}

Next, we compute the second factor $(1+\pi^{f-1},1+c)_{\ell^N}$.

We in fact calculate the factor for general $a\in\Q_{\ell}\x$.

\begin{propositioneigo}
Keep the notation of the previous theorem. Suppose $f=f(a)>0$. Then,
\[
(1+\pi^{f-1},1+c)_{\ell^N}
=\zeta_\ell^{-2c/\ell^w}
=
\begin{dcases}
\zeta_\ell^{-2c_0}
& (w=\ord_\ell(c))\\[2mm]
1
& (w<\ord_\ell(c))
\end{dcases}
\]
where $c_0:=\dfrac{c}{\ell^{\ord_{\ell}(c)}}\in\Z_\ell^\times$.
\end{propositioneigo}
\begin{proof}
    First, we consider the case when $w=\ord_\ell(c)$.
    In this case, by \cite[Proposition 4.15]{yanagihara2025rootnumberstwistedfermat} and \Cref{true conjecture}, we have
$$\frac{1}{\ell^{N-1}}[1 + \pi^{f-1}, 1+c]_{\ell^N} = 2c_0 \frac{(f-1)J(N,f)}{\ell^{N-w}}\equiv -2c_0 \pmod{\ell}.$$

    Next, we assume $w<\ord_\ell(c)$ and $c\neq 0$. In this case, $f>2\ell^{N-\ord_\ell(c)}$ by \Cref{sharifi}, and the proposition follows by \cite[Proposition 4.15]{yanagihara2025rootnumberstwistedfermat}.
\end{proof}

We eventually organize the results of the above section;

\begin{corollaryeigo}\label{value of hilbert symbol}
    Let $f=f(A),\ w=w(A)$. Suppose $f>0$. Then,
\[
(1+\pi^{f-1},A)_{\ell^N}=
\begin{dcases}
\zeta_\ell^{-2\left\{(N-1)r'+\dfrac{c}{\ell^w}\right\}}
& (w=\ord_\ell(b+c)=N-1)\\[2mm]
\zeta_\ell^{-2\dfrac{c}{\ell^w}}
& (\text{otherwise}).
\end{dcases}
\]
\end{corollaryeigo}





\section{Computation of the root number}

Based on the above calculations, we need to compute the root number $W(\phi_{\delta}^{(N)})$.
We first compute the relative root number $W(\phi_{\delta,\pi}, \eta_{\pi_F})$ by using Rohrlich's formula \cite[Proposition 2]{rohrlich1992root}. The strategy of the proof is almost the same as in \cite[Proposition~4.18]{yanagihara2025rootnumberstwistedfermat}.



\begin{propositioneigo}\label{relative}
   The conductor exponent of $\phi_{\delta,\pi}$ is given by
   \begin{align}
        f'
        &=\begin{dcases}
            2\ell^{N-w} & (1 \leq w < N,\ \ord_\ell(b + c) = w) \\
        \ell^{N-w-1}(\ell+1) & (1 \leq w < N,\ \ord_\ell(b + c) > w) \\
        2 & (w = N = \ord_\ell(c)) \\
            1 &(\text{otherwise}),
        \end{dcases}
    \end{align}
    and the relative root number is given by
    \begin{align}
        W(\phi_{\delta,\pi},\eta_{\pi_F})
        =\begin{dcases}
            -\leg{r'st\left\{(N-1)r'+\dfrac{c}{\ell^w}\right\}}{\ell} & (\ord_\ell(b + c) = w=N-1) \\
            -\leg{r'st\dfrac{c}{\ell^w}}{\ell} & \text{($\ord_\ell(b + c) = w<N-1$ or $w = N = \ord_\ell(c)$)} \\
            -\leg{-r'st\dfrac{c}{\ell^w}}{\ell} & \text{($1 \leq w < N$ and\ $\ord_\ell(b + c) >w$)} \\
            \leg{2}{\ell} &(\text{otherwise}).
        \end{dcases}
    \end{align}
            
\end{propositioneigo}
\begin{proof}
We first obtain the following from the Coleman--McCallum formula \cite[Theorem~5.3, Theorem~7.2]{coleman1988stable} (Be careful that the term $\dfrac{m^2}{3e}$ is a typo of $\dfrac{m^2}{3a}$ in \cite[Theorem~7.2]{coleman1988stable}):
	$$\phi_{\delta,\pi}(x)=(x,\delta)_{\ell^N}^{r+s} j_{r,s,t,\pi}(x)
	=(x,A)_{\ell^N}(x,\pi_{\rm{CM}})_2\ (x\in \mc{O}_{\Q_{\ell}(\zeta)}\x).$$

 Here, $\pi_{\rm{CM}}$ is a uniformizer of $\Q_{\ell}(\zeta_{\ell})$ that satisfies $\pi_{\rm{CM}}^{\ell-1}=-\ell$ and $\pi_{\rm{CM}}/(1-\zeta^{\ell^{N-1}})\equiv 1\pmod{\pi_{\rm{CM}}}$. 
 Write $\pi_{\rm{CM}}=u\pi^{\ell^{N-1}}$ for some unit $u\in \mc{O}_{\Q_{\ell}(\zeta)}\x$, then we find $(x,\pi_{\rm{CM}})_2=(x,u)_2(x,\pi)_2^{\ell^{N-1}}=(x,\pi)_2$ by the formula for tamely ramified Hilbert symbols \cite[Chapter V, Proposition 3.4]{neukirch2013algebraic}.

Next, Sharifi's formula gives

$$f' =
    \begin{dcases}
        \ell^{N-1}(\ell+1) & (w = 0) \\
        2\ell^{N-w} & (1 \leq w < N,\ \ord_\ell(b + c) = w) \\
        \ell^{N-w-1}(\ell+1) & (1 \leq w < N,\ \ord_\ell(b + c) > w) \\
        2 & (w = N = \ord_\ell(c)) \\
        1 & (\text{$w > N$ or $w = N \neq \ord_\ell(c)$}),
    \end{dcases}$$
    and the first assertion follows.
    
Next, we compute $W(\phi_{\delta,\pi},\eta_{\pi_F})$ using Rohrlich's formula \cite[Proposition 2]{rohrlich1992root}.

	If $f'=1$, we have $$W(\phi_{\delta,\pi},\eta_{\pi_F})=\leg{2}{\ell}.$$
	For $f'>1$,
	$$W(\phi_{\delta,\pi},\eta_{\pi_F})=\leg{-2l}{\ell}\phi_{\delta,\pi}(\pi)^{f'-1}i^u,$$
	where $u=\begin{dcases}
		0\ (\ell\equiv 1\pmod{4})\\
		1\ (\ell\equiv 3\pmod{4})
	\end{dcases}$

	and $l$ is the class determined by $\phi_{\delta,\pi}(1+\pi^{f'-1})=\zeta_{\ell}^l$.
        Since
	$$\phi_{\delta,\pi}(1+\pi^{f'-1})
	=(1+\pi^{f'-1},A)_{\ell^N},$$
	the value of $l$ is given by \Cref{value of hilbert symbol} as
    \[
l=
\begin{dcases}
-2\left\{(N-1)r'+\dfrac{c}{\ell^w}\right\}
& (w=\ord_\ell(b+c)=N-1)\\[2mm]
-2\dfrac{c}{\ell^w}
& (\text{otherwise}).
\end{dcases}
\]

For simplicity, we temporarily write $d:=\dfrac{\ell^{N-1}(\ell-1)}{2}$. By \cite[Lemma 4.17]{yanagihara2025rootnumberstwistedfermat} and \cite[Proposition 4]{rohrlich1992root}, we have
\begin{align}
	\phi_{\delta,\pi}(\pi)
=\phi_{\pi}(\pi)
=&\prod_{V\mid \infty} \phi_{V}(\pi)\inv\\
=&\prod_{h\in\Phi\inv} \frac{\pi^{\sigma_h}}{\abs{\pi^{\sigma_h}}}\\
=&i^{d}(-1)^{d}\cdot\left\{-\leg{(-1)^N\cdot 2r'st}{\ell}\right\}\inv\\
=&-i^{d}\leg{(-1)^{N+1}\cdot 2r'st}{\ell}.
\end{align}
Therefore, since $i^{u-d}=\leg{(-1)^{N}2}{\ell}$, we have
\begin{align}
    W(\phi_{\delta,\pi},\eta_{\pi_F})
=&\leg{-2l}{\ell}\left\{-i^{d(f'-1)}\leg{(-1)^{N+1} 2r'st}{\ell}\right\}i^u\\
=&-\leg{2r'stl}{\ell}i^{df'}\\
=&\begin{dcases}
            -\leg{-2r'stl}{\ell} & \text{(($1 \leq w < N$ and\ $\ord_\ell(b + c) = w$) or $w = N = \ord_\ell(c)$)} \\
        -\leg{2r'stl}{\ell} & (1 \leq w < N,\ \ord_\ell(b + c) > w) 
    \end{dcases},
\end{align}
noticing $\leg{-1}{\ell}=i^{\ell-1}$. Thus, the proposition follows.
\end{proof}

Combining the calculations so far, we obtain the following.
\begin{theoremeigo}\label{main}
Write
$$r^r s^s (\ell^N-t)^t \delta^{r+s}e = \epsilon \ell^{b}(1 + c)\ (\epsilon \in \mu_{\ell-1}(\Q_{\ell}),\ b \in \Z,\ c \in \ell \Z_{\ell})$$
where $e=\begin{dcases}
    1+\frac{3^N}{r'} &(\ell=3,\ N\geq 2)\\
    1 &(\text{otherwise})
\end{dcases}$. Let $w=\min \{\ord_{\ell}(b),\ord_{\ell}(c)\}$.

Then, the root number of $\phi_{\delta}^{(N)}$ is
$$W(\phi_{\delta}^{(N)})=\dprod_{\text{$p\le \infty$}}W_p(\phi_{\delta}^{(N)})$$
where, for $p\neq \ell$,
$$W_p(\phi_{\delta}^{(N)})
=\begin{dcases}
	i^{-\frac{\ell^{N-1}(\ell-1)}{2}}\ &(p=\infty)\\
	\leg{p}{\ell}\ &(p\mid \delta)\\
	1\ &(p\nmid \delta),
\end{dcases} $$
and
\begin{align}
        W_{\ell}(\phi_{\delta}^{(N)})
        =\begin{dcases}
            -\leg{r'st\left\{(N-1)r'+\dfrac{c}{\ell^w}\right\}}{\ell}i^{\frac{\ell^{N-1}(\ell-1)}{2}} & (\ord_\ell(b + c) = w=N-1) \\
            -\leg{r'st\dfrac{c}{\ell^w}}{\ell}i^{\frac{\ell^{N-1}(\ell-1)}{2}} & \text{($\ord_\ell(b + c) = w<N-1$ or $w = N = \ord_\ell(c)$)} \\
            -\leg{-r'st\dfrac{c}{\ell^w}}{\ell}i^{\frac{\ell^{N-1}(\ell-1)}{2}} & \text{($1 \leq w < N$ and\ $\ord_\ell(b + c) >w$)} \\
            \leg{2}{\ell}i^{\frac{\ell^{N-1}(\ell-1)}{2}} &(\text{otherwise}).
        \end{dcases}
    \end{align}
\end{theoremeigo}

\begin{exampleeigo}
    Consider the case $N=2,\ \ell=3$. For illustration, we consider the triples \((3,5,1)\), \((3,4,2)\), and \((6,2,1)\).

    The global root numbers $W(\phi_{\delta}^{(N)})$ for
    $1\leq\delta\leq20$ with $3\nmid\delta$ are as follows.

    \begin{table}[htbp]
  \centering
  \normalsize
  \setlength{\tabcolsep}{5pt}
  \renewcommand{\arraystretch}{1.25}
  \begin{tabular}{|c|*{14}{c|}}
    \hline
    $\delta$ & 1 & 2 & 4 & 5 & 7 & 8 & 10 & 11 & 13 & 14 & 16 & 17 & 19 & 20 \\
    \hline
    $\operatorname{ord}_{\ell}(b)$
      & 1 & 1 & 1 & 1 & 1 & 1 & 1 & 1 & 1 & 1 & 1 & 1 & 1 & 1 \\
    \hline
    $\operatorname{ord}_{\ell}(c)$
      & 1 & 2 & 1 & 1 & 4 & 1 & 1 & 2 & 1 & 1 & 2 & 1 & 1 & 3 \\
    \hline
    $\operatorname{ord}_{\ell}(b+c)$
      & 2 & 1 & 1 & 1 & 1 & 2 & 3 & 1 & 1 & 1 & 1 & 5 & 2 & 1 \\
    \hline
    $W(\phi_{\delta}^{(N)})$
      & $1$ & $-1$ & $1$ & $1$ & $1$ & $-1$ & $1$ & $-1$
      & $-1$ & $1$ & $-1$ & $-1$ & $1$ & $1$ \\
    \hline
  \end{tabular}
  \caption{The case $(r,s,t)=(3,5,1)$.}
\end{table}

\begin{table}[htbp]
  \centering
  \normalsize
  \setlength{\tabcolsep}{5pt}
  \renewcommand{\arraystretch}{1.25}
  \begin{tabular}{|c|*{14}{c|}}
    \hline
    $\delta$ & 1 & 2 & 4 & 5 & 7 & 8 & 10 & 11 & 13 & 14 & 16 & 17 & 19 & 20 \\
    \hline
    $\operatorname{ord}_{\ell}(b)$
      & 1 & 1 & 1 & 1 & 1 & 1 & 1 & 1 & 1 & 1 & 1 & 1 & 1 & 1 \\
    \hline
    $\operatorname{ord}_{\ell}(c)$
      & 1 & 1 & 2 & 3 & 1 & 1 & 1 & 1 & 2 & 2 & 1 & 1 & 1 & 1 \\
    \hline
    $\operatorname{ord}_{\ell}(b+c)$
      & 3 & 1 & 1 & 1 & 1 & 2 & 2 & 1 & 1 & 1 & 1 & 2 & 2 & 1 \\
    \hline
    $W(\phi_{\delta}^{(N)})$
      & $1$ & $1$ & $-1$ & $-1$ & $-1$ & $-1$ & $1$ & $1$
      & $1$ & $-1$ & $1$ & $-1$ & $1$ & $-1$ \\
    \hline
  \end{tabular}
  \caption{The case $(r,s,t)=(3,4,2)$.}
\end{table}

\begin{table}[htbp]
  \centering
  \normalsize
  \setlength{\tabcolsep}{5pt}
  \renewcommand{\arraystretch}{1.25}
  \begin{tabular}{|c|*{14}{c|}}
    \hline
    $\delta$ & 1 & 2 & 4 & 5 & 7 & 8 & 10 & 11 & 13 & 14 & 16 & 17 & 19 & 20 \\
    \hline
    $\operatorname{ord}_{\ell}(b)$
      & 1 & 1 & 1 & 1 & 1 & 1 & 1 & 1 & 1 & 1 & 1 & 1 & 1 & 1 \\
    \hline
    $\operatorname{ord}_{\ell}(c)$
      & 1 & 1 & 2 & 3 & 1 & 1 & 1 & 1 & 2 & 2 & 1 & 1 & 1 & 1 \\
    \hline
    $\operatorname{ord}_{\ell}(b+c)$
      & 4 & 1 & 1 & 1 & 1 & 2 & 2 & 1 & 1 & 1 & 1 & 2 & 2 & 1 \\
    \hline
    $W(\phi_{\delta}^{(N)})$
      & $1$ & $1$ & $-1$ & $-1$ & $-1$ & $-1$ & $1$ & $1$
      & $1$ & $-1$ & $1$ & $-1$ & $1$ & $-1$ \\
    \hline
  \end{tabular}
  \caption{The case $(r,s,t)=(6,2,1)$.}
\end{table}
\end{exampleeigo}

\section{Appendix}\label{Appendix1} In this section, we explain that when $N\geq 2$ and
    $(r,s,t)=(\ell^{N-1}(\ell-1),1,\ell^{N-1}-1)$, the new part
$\jac(C_{N})^{\mathrm{new}}$ 
of the Jacobian variety $\jac(C_{N})$ coincides with 
$\jac(C_{N})$ itself.  
We also check when $N=2$ in addition, that $C_{2}$ is isomorphic over $\Q$ to 
the non-singular projective plane 
curve $D_{\delta}:\delta^{s+2t}U W^\ell+U^{\ell+1}=\delta V^{\ell}W$ with the affine model 
\[
D^\circ_\delta:\quad \delta^{s+2t}u+u^{\ell+1}=\delta v^{\ell},\ 
(u,v)=\Big(\frac{U}{W},\frac{V}{W}\Big).
\]

As for the first claim, we observe  
$$y^{\ell^N}=x^{r}(\delta-x)^s=x^{\ell^N-s-t}(\delta-x)^s.$$
Multiplying both sides by $x^{-\ell^N}$, the equation becomes 
$$\Big(\frac{y}{x}\Big)^{\ell^N}=\Big(\frac{1}{x}\Big)^t \Big(\frac{\delta-x}{x}\Big)^s.$$
Substituting $(X,Y)=\Big(-\displaystyle\frac{\delta-x}{x},\frac{y}{x}\Big)$ with 
$\displaystyle\frac{1}{x}=\delta^{-1}(1-X)$, we have 
$$Y^{\ell^N}=(-1)^s\delta^{-t} X^s(1-X)^t.$$
Substituting $(X,Y)\mapsto ( \delta^{-1} X,(-1)^s Y)$, eventually, we see that  
$$C'^\circ_{N}:Y^{\ell^N}=\delta^{-s-2t} X^s(\delta-X)^t.$$
Let $C'_{N}$ be the non-singular projective model over $\Q$ of $C'^\circ_{N}$. Then, we have ${\rm Jac}(C_{N})\simeq {\rm Jac}(C'_{N})$ over $\Q$. Further, it is easy to see that 
$${\rm Jac}(C_{N})^{{\rm new}}\stackrel{\Q}{\simeq} {\rm Jac}(C'_{N})^{{\rm new}}$$
where ${\rm Jac}(C'_{N})^{{\rm new}}$ is the quotient of 
${\rm Jac}(C'_{N})$ by the Jacobian of the non-singular projective model of 
$$C'^\circ_{N-1}:Y^{\ell^{N-1}}=\delta^{-s-2t} X^s(\delta-X)^t.$$
Now, we specialize $(s,t)=(1,\ell^{N-1}-1)$. Then, the defining equation of $C'^\circ_{N-1}$ satisfies
$$\Big(\frac{Y}{\delta-X}\Big)^{\ell^{N-1}}=\frac{\delta^{-s-2t} X}{\delta-X}.$$
This means that $C'^\circ_{N-1}$ is a rational curve and thus, 
the Jacobian variety of its non-singular projective model is trivial. 
Therefore, we have the first claim. 

For the second claim, when $N=2$, substituting $(X,Y)=\Big(\displaystyle\frac{\delta^{s+2t}u-u^{\ell+1}+\delta v^\ell }{2v^\ell},\frac{u}{v}\Big)$, the affine curve $C'^\circ_{2}$ is birational over $\Q$ to 
$D^\circ_\delta$,  
since 
$$\delta^{-s-2t} X(\delta-X)^{\ell-1}=\delta^{-s-2t}\Big(\frac{\delta^{s+2t} u-u^{\ell+1}+\delta v^\ell }{2v^\ell}\Big)
\Big(\frac{-\delta^{s+2t}u+u^{\ell+1}+\delta v^\ell }{2v^\ell}\Big)^{\ell-1}$$
$$=\frac{u}{v^\ell} 
\Big(-\frac{u^{\ell+1}}{v^\ell}\Big)^{\ell-1}=\Big(\frac{u}{v}\Big)^{\ell^2}=Y^{\ell^2}$$
and the unique smooth projective models of both curves have the same genus $\displaystyle\frac{\ell^2-\ell}{2}$. Thus, we have the second claim.

\bibliographystyle{abbrv}
\bibliography{test}
\end{document}